\documentclass[11pt,a4paper]{amsart}
\usepackage{amsfonts}
\usepackage{amsthm}
\usepackage{amsmath}
\usepackage{amscd}
\usepackage{t1enc}
\usepackage{indentfirst}
\usepackage{graphicx}
\usepackage{pict2e}
\usepackage{epic}

\usepackage[colorlinks,citecolor=red,urlcolor=blue,bookmarks=false,hypertexnames=true]{hyperref} 

\numberwithin{equation}{section}
\usepackage[margin=2cm]{geometry}
\usepackage{epstopdf} 
\usepackage[utf8]{inputenc} 
\usepackage[T1]{fontenc}
\usepackage{array}
\usepackage{subcaption}

\usepackage{tikz}
\usetikzlibrary{graphs}
\usetikzlibrary{graphs.standard}

\theoremstyle{plain}
\newtheorem{Th}{Theorem}[section]

\newtheorem{Cor}[Th]{Corollary}
\newtheorem{Prop}[Th]{Proposition}

\theoremstyle{definition}

\newtheorem{Def}[Th]{Definition}

\newtheorem{?}[Th]{Problem}
\newtheorem{Ex}[Th]{Example}
\newtheorem*{Nt*}{Note}

\newcommand{\rk}{\operatorname {rk}}

\begin{document}

	\setcounter{page}{1}
	\vspace{2cm}

	\title{Minimum Generating Sets for Complete Graphs }
	\author{Selma Altınok and Gökçen Dilaver}

	\begin{abstract}
	Let $G$ be a graph with edges labeled by ideals of a commutative ring $R$ with identity. Such a graph is called an edge-labeled graph over $R$. A generalized spline is a vertex labeling so that the difference between the labels of any two adjacent vertices lies in the ideal corresponding to the edge. These generalized splines form a module  over $R$. In this paper, we consider complete graphs whose edges are labeled with proper ideals of $\mathbb{Z} / m\mathbb{Z}$. We compute  minimum generating sets of constant flow-up classes for spline modules on edge-labeled complete graphs over $\mathbb{Z} / m\mathbb{Z}$ and determine their rank  under some restrictions. \\

\end{abstract}
	
	\maketitle
	
	\section{Introduction} 
	
The term \emph{spline} was first used by engineers to model complicated objects like ships or cars. Mathematicians later adopted the term and they refer to such objects as "Classical splines". \emph{Classical splines} are piecewise polynomial functions defined on the faces of a polyhedral complex that
agree on the intersection of adjacent faces. Classical splines are important tools in topology and geometry to construct equivariant cohomology rings; applied mathematics to approximate complicated functions, and also recently in many areas like computer based animations and geometric design.

In this paper, we work on an extension of classical splines which was introduced  by Gilbert, Polster and Tymoczko \cite{2013}.
Let $R$ be a commutative ring with identity, $G=(V,E)$ be a  graph and $ \alpha : E \to \{\text{ideals in R}\}$ be a function which labels edges of $G$ by ideals of $R$. Such a graph $(G,\alpha)$ is called an edge-labeled graph. \emph{A generalized spline} is a vertex labeling on an edge-labeled graph $(G,\alpha)$ such that the difference between the labels of any two adjacent vertices lies in the corresponding edge ideal. The collection of generalized splines over the edge-labeled graph is denoted by $R_{(G, \alpha)}. $
$$R_{(G,\alpha)} = \{F \in R^{|V|} : \text{for each edge} \; e = uv, \; \text{the difference} \;	f_u - f_v \in \alpha(uv)\}.$$ 
It has a ring and an $R$- module structure. Such a module is called a generalized spline module.

Gilbert, Polster and Tymoczko \cite{2013} asked given $(G,\alpha)$ an edge-labeled graph or a particular family of graphs such as cycles, complete graph, bipartite graph or hypercubes could we find a minimal generating set for the generalized splines $R_{(G,\alpha)}$ when $R$ is an integral domain? Handschy, Melnick and Reinders \cite{2014} focus on integer
generalized splines on cycles. They presented a special type of generalized splines, called flow-up classes, and showed the existence of the smallest flow-up classes on cycles over $\mathbb{Z}$. Moreover, they proved that these flow-up classes formed a basis for generalized spline modules.  The same result holds for arbitrary graphs over integers by  Bowden, Hagen, King and Reinders \cite{2015bases}. 

Bowden and Tymoczko~\cite{2015} consider splines over $\mathbb{Z} / m\mathbb{Z}$. These $\mathbb{Z}$-modules of splines over $\mathbb{Z} / m\mathbb{Z}$ must have minimum generating sets, namely a generating set with the smallest possible size. Minimum generating sets can be smaller than expected. Over a domain, it is known that there are at least $n$ elements in the minimum generating set for graphs with n vertices. In case of splines over $\mathbb{Z} / m\mathbb{Z}$ there are at most $n$ elements (see Theorem 4.1 in~\cite{2015}). The such smallest number is called the rank of $\mathbb{Z}$-modules of splines over $\mathbb{Z} / m\mathbb{Z}$. Bowden and Tymoczko~\cite{2015}
gave an algorithm to produce minimum generating sets for splines on cycles. Later on, Philbin, Swift, Tammaro and Williams \cite{2017} extended Bowden and Tymoczko's work to arbitrary graphs and gave  an algorithm to produce minimum generating sets for connected graphs over $\mathbb{Z}/  m\mathbb{Z}$\@. 

In this paper, we determine an explicit minimal generating set for splines on complete graphs with $n$ vertices over $\mathbb{Z}/  m\mathbb{Z}$ under some restrictions on edge labelings and their rank $n$. Since  Philbin, Swift, Tammaro and Williams's algorithm is not applicable when  $m$ is sufficiently large or $G$ is a graph with a large number of vertices we use a different approach  to construct  minimum flow-up generating sets over $\mathbb{Z}/  m\mathbb{Z}$.
We use some results of  Altınok and Sarıoglan \cite{2019} to find flow-up generating sets for splines over $\mathbb{Z} / m\mathbb{Z}$\@. We show that  the obtained generating sets are in fact  minimum flow-up generating sets by using Criteria~\ref{criteria} in Section~\ref{modulo m} given by Bowden and Tymoczko \cite{2015}. As a result of this, we obtain a  minimum flow-up generating set on complete graphs whose edges are labeled by $a^i$ over $\mathbb{Z}/ m\mathbb{Z}$, in particular over $\mathbb{Z}/ p^t\mathbb{Z}$. Bowden and Tymoczko \cite{2015} also show that module of splines defined on a graph with  $n$ vertices over $\mathbb{Z}/ m\mathbb{Z}$ can have any rank between 2 and $n$. We give an alternative proof of this fact by using trails as considered in the work of Altınok and Sarıoglan \cite{2019}, and show  existence of a complete graph $K_n$ over $\mathbb{Z} / m\mathbb{Z}$ whose rank can be between 1 and $n$ where $m$ is the product of two distinct prime numbers and $n \ge 4$.

	\section{Background and Notations}
	
	In this section, we introduce some important definitions and notations for generalized splines. Throughout this paper, we assume that $R$ is a commutative ring with identity and $G=(V,E)$ is a finite simple connected  graph, defined by a set of vertices $V$ and a set of edges $E$\@.  

We use $K_n$ to denote a complete graph with $n$ vertices, $C_n$ to denote a cycle graph with $n$ vertices, $S_n$ to denote a star graph with $n+1$ vertices, and $P_n$ to denote a path graph with $n$ vertices. 

\begin{Def}
	Let $G = (V,E)$ be a graph and let $R$ be a commutative ring with identity. An edge-labeling function is a map
	$ \alpha : E \to \{\text{ideals in R}\}$ from the set of edges of $G$ to the set of nonzero ideals in $R$\@. We call the pair \begin{math}
		(G, \alpha)
	\end{math}  an \emph{edge-labeled graph}\@.
\end{Def}
\begin{Def}
	Let ($G$,\begin{math}
		\alpha \end{math}) be an edge-labeled graph. \emph{A generalized spline} on $(G,\alpha)$ is a vertex labeling 
	$F \in R^{\mid V \mid }$ such that for each edge 
	$e_{ij}= v_iv_j \in E$ we have $f_{v_i} - f_{v_j} \in \alpha(v_iv_j)$, where $f_{v_i}$ denotes the label on the vertex $v_i$ \@.
	The collection of generalized splines on $(G,\alpha)$ with base ring $R$ is denoted by $R_{(G, \alpha)} $ and for elements of $R_{(G, \alpha)} $ we use column matrix notation with ordering from bottom to top as follows: 
	
	\centering \begin{displaymath}
		F =
		\left( \begin{array}{c}
			f_{v_n} \\
			\vdots\\
			f_{v_1}
		\end{array} \right) \in R_{(G, \alpha)}.
	\end{displaymath} 
	\raggedright We also use the vector notation  \begin{math}
		F= (f_{v_1},\ldots,f_{v_n}).
	\end{math}
	In the following we refer to generalized splines simply as splines. 
\end{Def} 

In this paper, our base ring is $R = \mathbb{Z}/ m\mathbb{Z}$\@. Since every ideal in $\mathbb{Z}/ m\mathbb{Z}$ is principal, we represent the edge label by the positive smallest integer  $f$ in its coset $\bar{f} = f+m\mathbb{Z}$ where the ideal corresponding to the edge  is $I=\langle \bar f\rangle$.  We generally assume that the edges of our graphs are not labeled by trivial ideals unless otherwise stated.

It can be easily observed that
\begin{itemize}
	\item $R_{(G,\alpha)}$ is a ring with unit 1 defined by $1_v = 1 $ for each vertex $v \in V $\@.
	\item $R_{(G,\alpha)}$ has an $R$-module structure with componentwise addition and multiplication by elements of $R$\@.
	\item If $R$ is a principal ideal domain, then  $R_{(G,\alpha)}$ has a free $R$-module structure.
\end{itemize}
\begin{Def}
	Let  $(G,\alpha)$  be an edge-labeled graph with an ordered set of vertices $\{v_1,\ldots,v_n\}$\@. \emph{An $i$-th flow-up class} $F^{(i)}$ over \begin{math}
		(G,\alpha) \end{math}, for $1\le i \le n $, 
	is a spline for which the component $f_{v_i}^{(i)} \neq 0 $ and $f_{v_t}^{(i)} = 0 $ whenever $t < i $\@. The set of all $i$-th  flow-up classes is denoted by $\mathcal{F}_i$\@.
\end{Def}


\begin{Def}
	\emph{A constant flow-up class} in $R_{(G,\alpha)}$ is a flow-up spline $F^{(i)}$ for which there exists an element $a_i \in R$  such that $f_{v_j}^{(i)} \in \{0,a_i\}$ for each vertex $v_j \in V$.
\end{Def}


When $R=\mathbb Z/m\mathbb Z$, we use the notation $[\mathbb Z/m\mathbb Z]_{(G,\alpha)}$ for $R_{(G,\alpha)}$ as a $\mathbb  Z$-module.	We define a minimum generating set for a $\mathbb Z$-module $[\mathbb Z/m\mathbb Z]_{(G,\alpha)}$ as follows.
\begin{Def}
	\emph{A minimum generating set} for a $\mathbb{Z}$-module $[\mathbb Z/m\mathbb Z]_{(G,\alpha)}$ is a spanning set of splines with the smallest possible number of elements. The size of a minimum generating set is called  the \emph{rank} and is denoted by $\rk [\mathbb Z/m\mathbb Z]_{(G,\alpha)}$.
\end{Def}


\begin{figure}[h]
	\centering	
	\begin{subfigure}[b]{0.32\textwidth}
		\includegraphics[width=1\textwidth]{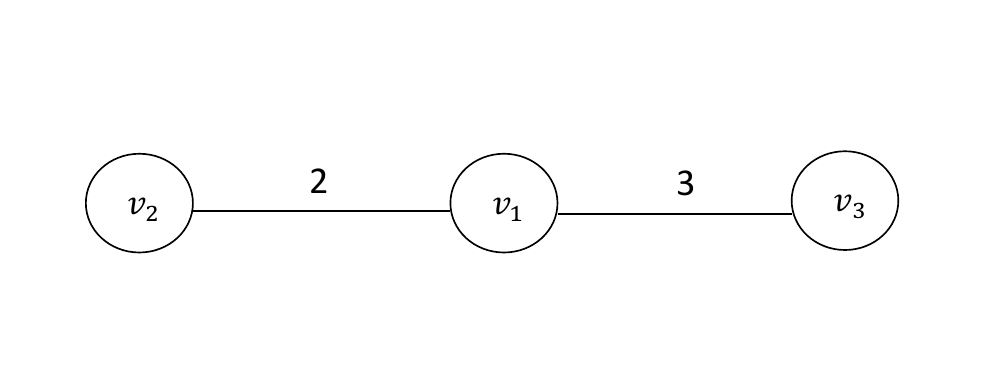}
		\caption{Path graph $P_3$}
		\label{fig:deneme}
	\end{subfigure} 
	\:\:\:\:\:
	\begin{subfigure}[b]{0.25\textwidth}
		\includegraphics[width=1\textwidth]{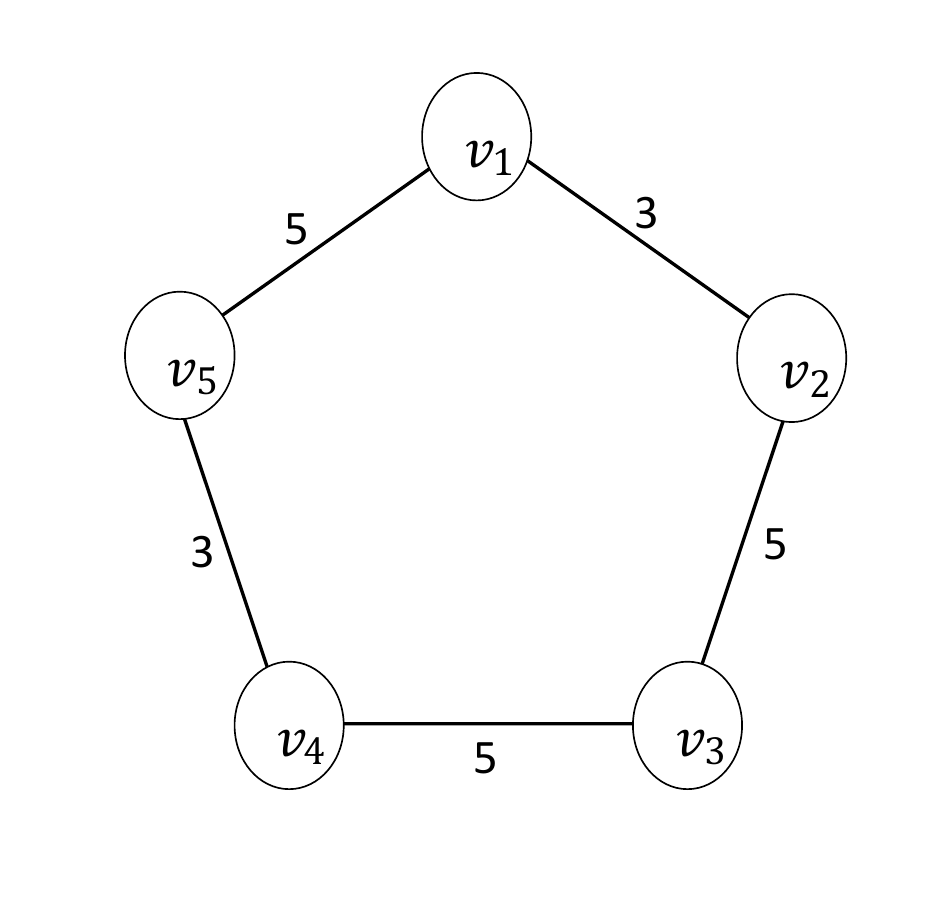}
		\caption{Cycle graph $C_5$}
		\label{fig:deneme1}
	\end{subfigure} 
	\caption{}
\end{figure}	
\begin{Ex}
	Let \begin{math}
		(P_3,\alpha) \end{math} be an edge-labeled path graph with three vertices over $\mathbb{Z}/ 6\mathbb{Z}$ as in Figure \ref{fig:deneme}. Then a flow-up generating set for $[\mathbb{Z}/ 6\mathbb{Z}]_{(P_3,\alpha)}$  can be given by $$B= \{(1,1,1),(0,2,3),(0,0,3)\}$$
	Since  $B$ is not linearly independent, for example, $3.(0,2,3) \equiv (0,0,3) $, it is not minimal and hence can not be a minimum generating set. Therefore, a minimum generating set for $[\mathbb{Z}/ 6\mathbb{Z}]_{(P_3,\alpha)}$ is $$B=\{(1,1,1),(0,2,3)\}$$
	so that the rank of $[\mathbb{Z}/ 6\mathbb{Z}]_{(P_3,\alpha)}$ is $2$\@.
\end{Ex}

\begin{Ex}
	Let \begin{math}
		(C_5,\alpha) \end{math} be an edge-labeled cycle graph with five vertices over $\mathbb{Z}/ 15\mathbb{Z}$ as in Figure \ref{fig:deneme1}. Then a flow-up generating set for $[\mathbb{Z}/ 15\mathbb{Z}]_{(C_5,\alpha)}$ is given by $$B=\{(1,1,1,1,1),(0,3,3,3,0),(0,0,5,5,5),(0,0,0,5,5)\}.$$
	The set $B$ is minimal, but it is not a minimum generating set. A minimum generating set for $[\mathbb{Z}/ 15\mathbb{Z}]_{(C_5,\alpha)}$ is given as $$B=\{(1,10,0,0,6),(0,6,1,6,0),(0,0,0,10,10)\}$$ by Theorem 3.15 in \cite{2017}.
\end{Ex}


\begin{Def}\label{homo}
	Let $\rho: R^{'} \to R$ be a ring homomorphism and $(G,\alpha)$ be an edge-labeled graph over $R.$ Suppose that  $ \rho ^{-1}(\alpha(e))$ is an ideal of $R^{'}$. The map $$\rho_* : R^{'}_{(G, \rho^{-1}(\alpha))} \to R_{(G,\alpha)}$$ is the restriction of the product map $\rho_* : (R^{'})^{\mid V \mid} \to R^{\mid V \mid}$ to the rings of splines namely $$(\rho_*f)_v = \rho (f_v)$$ for each spline $f \in R^{'}_{(G,\rho^{-1}(\alpha))}$ and each vertex $v \in V.$ The following can be easily observed.
	
	\begin{itemize}
		\item The map $\rho_* : R^{'}_{(G, \rho^{-1}(\alpha))} \to R_{(G,\alpha)}$ is a well-defined ring homomorphism.
		\item If  $\rho: R^{'} \to R$ is an injection then the map $\rho_* : R^{'}_{(G, \rho^{-1}(\alpha))} \to R_{(G,\alpha)}$ is an injection.
		\item If $\rho: R^{'} \to R$ is a surjection then the map $\rho_* : R^{'}_{(G, \rho^{-1}(\alpha))} \to R_{(G,\alpha)}$ is a surjection.
	\end{itemize}
	
\end{Def}

\begin{Ex}
	
	The natural surjection map $ \rho : \mathbb{Z} \to \mathbb{Z}/m\mathbb{Z}$ induces a surjection $$ \rho_* : [\mathbb{Z}]_{(G,\rho^{-1}(\alpha))} \to [\mathbb{Z}/m\mathbb{Z}]_{(G,\alpha)}$$ given by $(\rho_*f)_{v_i}= \rho(f_{v_i}) $ for each spline $f \in [\mathbb{Z}]_{(G,\rho^{-1}(\alpha))}$ and each $v_i \in V$. In particular, if $\{a_1,\ldots,a_n\}$ is a basis for the $\mathbb{Z}$-module of splines $[\mathbb{Z}]_{(G,\rho^{-1}(\alpha))}$  then $\{\rho_*(a_1),\ldots,\rho_*(a_n)\}$ spans the $\mathbb{Z}$-module of splines $[\mathbb{Z}/m\mathbb{Z}]_{(G,\alpha)}$.
	As a result of this, we identify each  spline $\bar F=(\bar f_{v_1},\ldots,\bar f_{v_n}) \in [\mathbb{Z}/m\mathbb{Z}]_{(G,\alpha)}$ 
	with any integer spline $F=(f_{v_1},\ldots,f_{v_n})$, 
	where  $f_{v_i}$ is the smallest nonnegative integer coset representative of  $\bar f_{v_i}$ 
	for $i=1,\ldots,n$.
\end{Ex}	
	
	\subsection{Edge-labels on Complete Graphs} \label{completegraph}
We consider a complete graph $K_n$ together with a star graph $S_n$\@. We recursively obtain a complete graph $K_{n+1}$  as $K_{n+1} = K_n + S_n$  as follows. We start with $n=3$. We label the vertices of $K_3$ as $v_1, v_2,v_3$ in clockwise order and label the vertices of $S_3$ starting from $v_1$ to $v_3$ in clockwise order with the central vertex $v_4$. We add a star graph $S_3$ to $K_3$ in order to obtain a complete graph $K_4$ by joining the corresponding vertices of $K_3$ and $S_3$, simply  $K_4 = K_3 + S_3$ as in Figure \ref{fig:k4.pdf}. Then we inductively define a complete graph $K_{n+1}$  as $K_{n+1} = K_n + S_n$ for a complete graph $K_n$ with vertices $v_1,\ldots,v_n$  and a start graph $S_n$ starting from $v_1$ to $v_{n}$ in clockwise order and the central vertex $v_{n+1}$\@.

\begin{figure}[h!]
	\centering
	\includegraphics[width=0.65\linewidth]{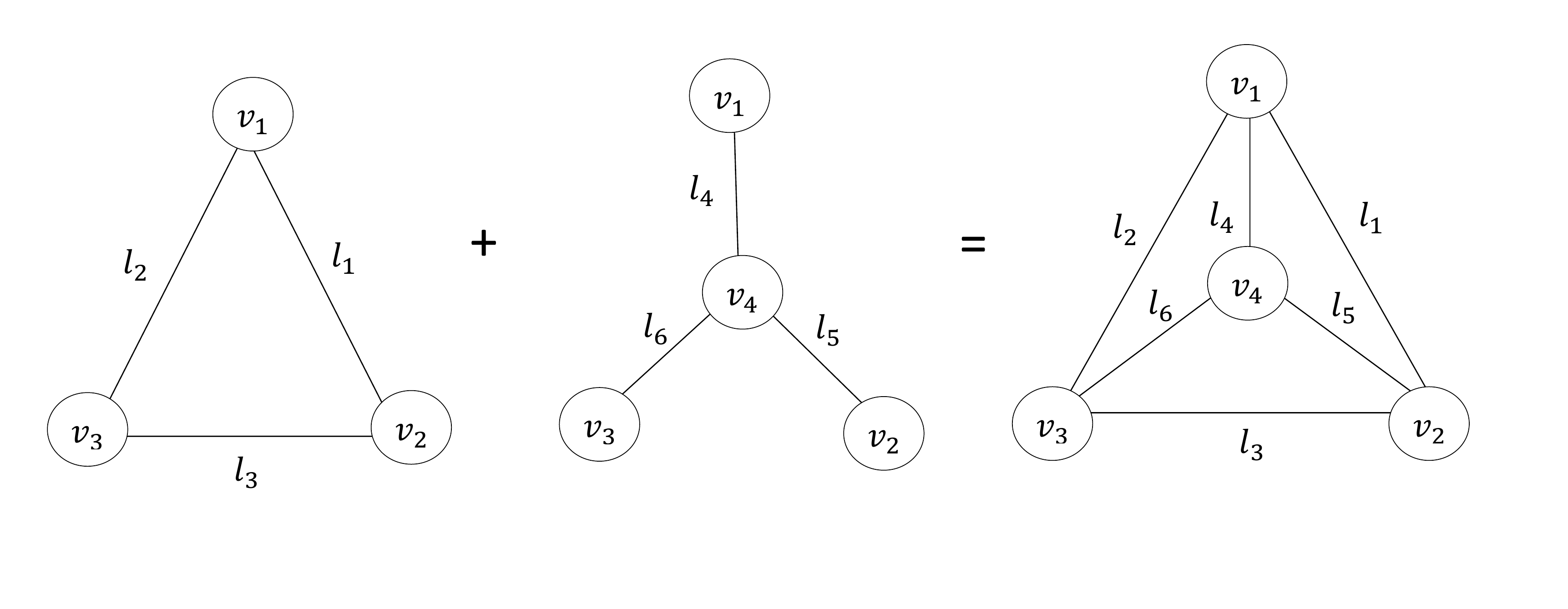}
	\caption{Construction of the edge-labeled complete graph $ K_3 + S_3 = K_4 $}
	\label{fig:k4.pdf}
\end{figure}

We fix $r_n= \frac{n(n-1)}{2}$ for any $n$. When $n \ge 1$ the number $r_n$ corresponds to the number of the edges of $K_n$\@. 		In this paper, we consider $K_n$ for $n \ge 3$, unless otherwise stated.
We want to label the edges of complete graphs as follows. We denote the edges $e_{ij}=v_iv_j$ by $e_k$ as follows:  We set the edges of $K_3$ in order by $e_1:=e_{12}, e_2:=e_{13}, e_3:=e_{23}$ and label them by generators $$l_1 \in \alpha(e_1), l_2 \in \alpha(e_2), l_3 \in \alpha(e_3).$$ We  denote the edges of $S_3$ by $e_4:=e_{14}, e_5:=e_{24}, e_6:=e_{34}$ and label them by $$l_4 \in \alpha(e_4), l_5 \in \alpha(e_5), l_6 \in \alpha(e_6).$$
Then we add $S_3$ to $K_3$ to obtain $K_4$ with the edges labeled by 
$$l_1 \in \alpha(e_1), l_2 \in \alpha(e_2), l_3 \in \alpha(e_3),
l_4 \in \alpha(e_4), l_5 \in \alpha(e_5), l_6 \in \alpha(e_6)$$ as in Figure \ref{fig:k4.pdf}.
For the construction of an edge-labeled graph $K_5$, we start by denoting the edges of $S_4$ as $e_7:=e_{15}, e_8:=e_{25}, e_9:=e_{35}, e_{10}:=e_{45}$ and label them by $$l_7 \in \alpha(e_7), l_8 \in \alpha(e_8), l_9 \in \alpha(e_9), l_{10} \in \alpha(e_{10}).$$ Then we add $S_4$ to $K_4$ to obtain $K_5$ with the edges labeled by 
$$l_1 \in \alpha(e_1), l_2 \in \alpha(e_2), l_3 \in \alpha(e_3), l_4 \in \alpha(e_4), l_5 \in \alpha(e_5), l_6 \in \alpha(e_6),$$
$$l_7 \in \alpha(e_7), l_8 \in \alpha(e_8), l_9 \in \alpha(e_9), l_{10} \in \alpha(e_{10}).$$
If we continue this way, we can recursively construct $K_n = K_{n-1} + S_{n-1}$ with the edges labeled by
$$l_1 \in \alpha(e_1), l_2 \in \alpha(e_2), l_3 \in \alpha(e_3), \ldots, l_{r_{n-1}-1} \in \alpha(e_{r_{n-1}-1}), l_{r_{n-1}} \in \alpha(e_{r_{n-1}}),$$
$$  l_{r_{n-1}+1} \in \alpha(e_{r_{n-1}+1}), l_{r_{n-1}+2} \in \alpha(e_{r_{n-1}+2}), \ldots,  l_{r_n} \in \alpha(e_{r_n}).$$



\section{Minimum generating sets modulo $m$} \label{modulo m}
In this section, we determine  minimum flow-up generating sets for complete graphs over $\mathbb{Z}/ m\mathbb{Z}$ under some specific conditions. In order to do this, we first use some results of Altınok and Sarıoglan \cite{2019} to find flow-up generating sets for splines over $\mathbb{Z}/ m\mathbb{Z}$\@. Then we use a criterion given by Bowden and Tymoczko \cite{2015} to show that the generating sets obtained are in fact minimum flow-up generating sets. 

Throughout this section, we suppose that $p$ and $p_i$, where $i=1,\ldots, k$, are prime numbers and $m = p_1^{m_1}p_2^{m_2}\cdots p_k^{m_k}$ is a primary decomposition. Let  $a_i= p_1^{n_{i1}}p_2^{n_{i2}}\ldots p_k^{n_{ik}}$ be a zero divisor in $\mathbb{Z}/ m\mathbb{Z}$, where $0 \le n_{ij} \le m_j$, $i=1,2,\ldots, r_n$ and $j=1,2,\ldots,k$. We assume that all the edges of $K_n$ are ordered as in Section \ref{completegraph} and are labeled by 
$$l_1=a_1=p_1^{n_{11}}p_2^{n_{12}}\ldots p_k^{n_{1k}}, $$ 
$$ l_2=a_2=p_1^{n_{21}}p_2^{n_{22}}\ldots p_k^{n_{2k}},  $$ 
$$\vdots$$
$$l_{r_n}= a_{r_n}=p_1^{n_{r_n1}}p_2^{n_{r_n2}}\ldots p_k^{n_{r_nk}}. $$


Bowden, Hagen, King and Reinders \cite{2015bases} gave a characterization of flow-up bases for splines over the integers. Here we modified their theorem to get a set of flow-up generators over $\mathbb{Z}/ m\mathbb{Z}$.

\begin{Th}\label{basis}
	Let $(G,\alpha)$ be an edge-labeled graph over  $\mathbb{Z}/ m\mathbb{Z}$. The following are equivalent:
	\begin{itemize}
		\item The set $\{\bar{F}^{(1)},\ldots,\bar{F}^{(n)}\}$ forms a flow-up generator for $[\mathbb{Z}/ m\mathbb{Z}]_{(G,\alpha)}$\@.
		\item For each flow-up class $\bar{A}^{(i)} =(\bar{0},\ldots,\bar{0},\bar{g}_{v_i},\ldots,\bar{g}_{v_n})$ the entry $\bar{g}_{v_i}$ is a multiple of the entry $\bar{f}_{v_i}^{(i)}$\@.
	\end{itemize}

\end{Th}

\begin{proof}	
	Suppose that the set $\{\bar{F}^{(1)},\ldots,\bar{F}^{(n)}\}$ forms a flow-up generator for $[\mathbb{Z}/ m\mathbb{Z}]_{(G,\alpha)}$. Let  $\bar{A}^{(i)} =(\bar{0},\ldots,\bar{0},\bar{g}_{v_i},\ldots,\bar{g}_{v_n})$ be a spline in $[\mathbb{Z}/ m\mathbb{Z}]_{(G,\alpha)}$ with $i-1$ leading zeros. We can write $\bar{A}^{(i)}$ as a linear combination of the splines $\bar{F}^{(1)},\ldots,\bar{F}^{(n)}$ because they form a  flow-up generator. Since $\bar{A}^{(i)}$ has $i-1$ leading zeros, $\bar{A}^{(i)} = \bar{a}_i\bar{F}^{(i)} + \cdots + \bar{a}_n\bar{F}^{(n)} $ for some $\bar{a}_i,\ldots,\bar{a}_n \in \mathbb{Z}/m\mathbb{Z}$\@. It follows that the $i$-th entry of $\bar{A}^{(i)}$ is $\bar{g}_{v_i} = \bar{a}_i\bar{f}_{v_i}^{(i)} +\bar{a}_{i+1}\bar{0}+ \cdots+ \bar{a}_n\bar{0} $, that is $\bar{g}_{v_i} = \bar{a}_i\bar{f}_{v_i}^{(i)} $, for some $\bar{a}_i \in \mathbb{Z}/m \mathbb{Z}$\@.

	\hfil\break
	We now prove the converse. Assuming that for each flow-up class 
	$\bar{A}^{(i)} =(\bar{0},\ldots,\bar{0},\bar{g}_{v_i},\ldots,\bar{g}_{v_n})$ the entry $\bar{g}_{v_i}$ is a multiple of the entry $\bar{f}_{v_i}^{(i)}$ of any corresponding flow-up class $\bar{F}^{(i)}$, we show that  the  set $\{\bar{F}^{(1)},\ldots,\bar{F}^{(n)}\}$ forms a flow-up generator for $[\mathbb{Z}/ m\mathbb{Z}]_{(G,\alpha)}$. 
	Let $\bar{A} = (\bar{h}_{v_1},\ldots,\bar{h}_{v_n})$ be an arbitrary spline in  $[\mathbb{Z}/ m\mathbb{Z}]_{(G,\alpha)}$\@. We will show by a finite induction that for each $j\in \{1,2,\ldots,n\}$ we can write		 
	$$\bar{A} = \bar{A}_j^{'} + \sum_{k=1}^{j} \bar{a}_k\bar{F}^{(k)}, $$
	where $\bar{A}_j^{'}$ is a spline with (at least) $j$ leading zeros. The base case is when $j=1$. By assumption we have $\bar{h}_{v_1} = \bar{a}_1\bar{f}_{v_1}^{(1)}$. Therefore  $\bar{A}$ as
	$$ \bar A= (\bar{0},\bar{h}_{v_2} -\bar{a}_1\bar{f}_{v_2}^{(1)},\ldots,\bar{h}_{v_n} -\bar{a}_1\bar{f}_{v_n}^{(1)}) 	+ \bar{a}_1\bar{F}^{(1)},$$
	where ~$\bar{A}_1^{'} = (\bar{0},\bar{h}_{v_2} - \bar{a}_1\bar{f}_{v_2}^{(1)},\ldots,\bar{h}_{v_n} - \bar{a}_1\bar{f}_{v_n}^{(1)})$ is a spline.
	Now  assume, as the induction hypothesis, that  the equality $\bar{A} = \bar{A}_{j-1}^{'} + \sum\limits_{k=1}^{j-1} \bar{a}_k\bar{F}^{(k)} $ holds for some $j\in \{2,3\ldots,n\}$, where $\bar{A}_{j-1}^{'}$ is a spline with (at least) $j-1$ leading zeros.	We show that the same equality holds with $j-1$ replaced by $j$. Thus suppose that $$\bar{A}= (\bar{0},\ldots,\bar{0},\bar{h}_{v_j}^{'},\ldots,\bar{h}_{v_n}^{'}) + \sum\limits_{k=1}^{j-1} \bar{a}_k\bar{F}^{(k)}.$$	
	Since  $\bar{A}_{j-1}^{'}$ is a flow-up class, by assumption, we have $\bar{h}_{v_j}^{'} = \bar{a}_j\bar{f}_{v_j}^{(j)}$ for some $\bar{a}_j \in \mathbb{Z}/m\mathbb{Z}$. It follows that $\bar{A}$ can be written by
	$$ A =
	(\bar{0},\ldots,\bar{0},\bar{h}_{v_{j+1}}^{'} - \bar{a}_j\bar{f}_{v_{j+1}}^{(j)},\ldots,\bar{h}_{v_n}^{'} - \bar{a}_j\bar{f}_{v_n}^{(j)}) + \sum\limits_{k=1}^{j} \bar{a}_k\bar{F}^{(k)}.$$ 
	By letting $\bar{A}_j^{'}=(\bar{0},\ldots,\bar{0},\bar{h}_{v_{j+1}}^{'} - \bar{a}_j\bar{f}_{v_{j+1}}^{(j)},\ldots,\bar{h}_{v_n}^{'} - \bar{a}_j\bar{f}_{v_n}^{(j)})$, we obtain  $\bar{A} = \bar{A}_j^{'} + \sum\limits_{k=1}^{j} \bar{a}_k\bar{F}^{(k)} $ where $\bar{A}_j^{'}$ is a spline with $j$ leading zeros.
	In particular, we have  $\bar{A} = \bar{A}_n^{'} + \sum\limits_{k=1}^{n} \bar{a}_k\bar{F}^{(k)} $\@. Since $\bar{A}_n^{'}=(\bar{0},\ldots,\bar{0})$ is a spline with $n$ leading zeros, we obtain  $\bar{A} = \sum\limits_{k=1}^{n} \bar{a}_k\bar{F}^{(k)} $\@.
	
\end{proof}



The following criterion by Bowden and Tymoczko \cite{2015} gives conditions for when the set of flow-up generators is in fact a minimum flow-up generating set over $\mathbb{Z}/ m\mathbb{Z}$.  

\begin{Cor}{\upshape\cite{2015}}\label{criteria}
	Suppose that $\{b_1,b_2,\ldots,b_k\}$ is a set of flow-up generators for $[\mathbb{Z}/ m\mathbb{Z}]_{(G,\alpha)}$ satisfying the following properties:
	\begin{itemize}
		\item The spline $b_1 = (1,\ldots, 1)$.
		\item The splines $\{b_i : i = 2,3,\ldots,k\}$ are constant flow-up classes satisfying ~$b_{i_v} \in \{0,c_i\}$ for each $v \in V$ and each $i$.
		\item The set $\{c_1=1,c_2,\ldots,c_k\}$ can be reordered so that $c_{i_1} \mid c_{i_2} \mid \cdots \mid c_{i_k}$.
		
	\end{itemize}
	
	Then $\{b_1,b_2,\ldots,b_k\}$ forms a minimum flow-up generating set for $[\mathbb{Z}/ m\mathbb{Z}]_{(G,\alpha)}$.
	
\end{Cor}

\begin{proof}
	See Corollary 2.11 in \cite{2015}.
\end{proof}



\begin{Def}
	Let $(G,\alpha)$ be an edge-labeled graph with $n$ vertices. A \emph{trail} is a sequence of vertices and edges, i.e., $v_1,e_{12},v_2,\ldots,v_{k-1}e_{k-1k},v_k $ in which no edge is repeated. A trail $p^{(i,j)}$ that connects $v_i$ to a vertex $v_j$ is called a $v_j$-trail of $v_i$.

\end{Def}

Let $p^{(i,j)}$  be a $v_j$-trail of $v_i$\@. We use the notation $(p^{(i,j)})$ to refer to a greatest common divisor of the edge-labels on $p^{(i,j)}$\@  or, for brevity, a greatest common divisor of a $v_j$-trail of $v_i$. The set $\{(p^{(i,j)})\}$ denotes the set of greatest common divisors of all $v_j$-trails of $v_i$. Square brackets ``$[\, ]$'' refer  to a least common multiple of a set.  Note that since $(p^{(i,j)})=(p^{(j,i)})$ we only consider $v_j$-trails of $v_i$ when $j<i$ in this paper.

In general, we may not be able to construct a flow-up class $F^{(i)}=(0,\ldots,0,f_{v_i}^{(i)},\ldots,f_{v_n}^{(i)}) \in \mathcal{F}_i$ with the first nonzero entry  $f_{v_i}^{(i)} = [\cup_{k=1}^{i-1} \{(p^{(i,k)})\}]$ \@. If such a flow-up class exists, then $f_{v_i}^{(i)}$  is called \emph{the smallest nonzero leading entry} of the elements of $\mathcal{F}_i$\@. Altınok and Sarıoglan~\cite{2019} proved that if the base ring $R$ is a principal ideal domain (PID),  it can always be constructed.

Theorem~ \ref{min} and \ref{exist} are modifications of Altınok and Sarıoglan's results on constructing a flow-up class  $\bar{F}^{(i)}=(\bar{0},\ldots,\bar{0},\bar{f}_{v_i}^{(i)},\ldots,\bar{f}_{v_n}^{(i)}) $ over  $\mathbb{Z}/ m\mathbb{Z}$ (see Corollary~3.4 and Theorem~3.8 in \cite{2019}). As we mentioned before, our base ring is $R = \mathbb{Z}/ m\mathbb{Z}$\@. Since every ideal in $\mathbb{Z}/ m\mathbb{Z}$ is principal, we represent the edge label by the smallest integer  $f$ in its coset $\bar{f} = f+m\mathbb{Z}$ where the corresponding ideal to an edge  is $I=\langle\bar f \rangle$. As a result of this, we can work with $(G, \alpha)$ over $\mathbb{Z}$ with edge ideals represented the smallest elements in cosets and then pass to $\mathbb{Z}/ m\mathbb{Z}$. Working over $\mathbb{Z}$ allows one to  use the greatest common and the least common divisor.  By a greatest common divisor $(p^{(i,j)})$ of a $v_j$-trail of $v_i$ on $(G,\alpha)$ over $\mathbb{Z}/ m\mathbb{Z}$ we mean  the greatest common integer of a $v_j$-trail of $v_i$ over $\mathbb Z$ whose edge labels are represented by  the smallest positive elements in $\mathbb Z$.

\begin{Th}\label{min}
	Let $(G,\alpha)$ be an edge-labeled graph with $n$ vertices over  $\mathbb{Z}/m\mathbb{Z}$ and $v_j$ be a vertex with $ j \ge i $.  Let $\bar{F}^{(i)}=(\bar{0},\ldots,\bar{0},\bar{f}_{v_i}^{(i)},\ldots,\bar{f}_{v_n}^{(i)}) \in \mathcal{F}_i$ be an i-th flow-up class with $i >1$. Let $\{(p^{(j,k)})\}$ be the set of  greatest common divisors of all  $v_k$-trails $p^{(j,k)}$ of $v_j$. If $[\cup_{k=1}^{i-1}\{(p^{(j,k)}) \}]\not\equiv 0 \mod m \mathbb{Z}$, then the entry $\bar{f}_{v_j}^{(i)}$ is a multiple of the number of $[\cup_{k=1}^{i-1}\{(p^{(j,k)}) \}]$ modulo $m$.
\end{Th}

\begin{proof} Let $\bar{F}^{(i)}=(\bar{0},\ldots,\bar{0},\bar{f}_{v_i}^{(i)},\ldots,\bar{f}_{v_n}^{(i)}) \in \mathcal{F}_i$ be an i-th flow-up class with $i >1$. This gives a flow-up class $F^{(i)} = (0,\ldots,0,f_{v_i}^{(i)},\ldots,f_{v_n}^{(i)})$ on $(G,\alpha)$ over $\mathbb Z$ where  $f_{v_t}^{(i)} \in \mathbb{Z}$ is the smallest nonnegative coset representative of  $\bar{f}_{v_t}^{(i)}$ for $t=i,\ldots,n$. By Corollary~3.4 in \cite{2019}, $[\cup_{k=1}^{i-1}\{(p^{(j,k)}) \}]$  divides $f_{v_j}^{(i)}$ over $\mathbb{Z}$, so there exists $a \in \mathbb{Z}$ such that $$f_{v_j}^{(i)} =a[\cup_{k=1}^{i-1}\{(p^{(j,k)}) \}].$$ 
	Hence  $ \bar{f}_{v_j}^{(i)} =\bar{a}\overline{[\cup_{k=1}^{i-1}\{(p^{(j,k)}) \}]}$. 
	
\end{proof}


\begin{Th}\label{exist}
	Let $(G,\alpha)$ be an edge-labeled graph with $n$ vertices over  $\mathbb{Z}/m\mathbb{Z}$.  Let $v_i$ be a vertex of $G$ with $i>1$ and the vertex labels  $\bar{f}_{v_j}^{(i)}=\bar{0}$ for all $j<i$ over $\mathbb Z/m\mathbb{Z}$.
	Suppose that $f_{v_i}^{(i)} =[\cup_{k=1}^{i-1}\{(p^{(i,k)}) \}]\not\equiv 0 \mod m$, where $\bar{f}_{v_i}^{(i)} = f_{v_i}^{(i)} +  m\mathbb{Z}$ \@. Then there exists a flow-up class $$\bar{F}^{(i)}=(\bar{0},\ldots,\bar{0},\bar{f}_{v_i}^{(i)},\ldots,\bar{f}_{v_n}^{(i)}) \in \mathcal{F}_i.$$
\end{Th}

\begin{proof}
	We choose the smallest non-negative representatives $f_{v_j}^{(i)}=0 $ in $\mathbb Z$ of the vertex labels $\bar{f}_{v_j}^{(i)}=\bar{0}$ for all $j<i$. Let $f_{v_i}^{(i)} = [\cup_{k=1}^{i-1}\{(p^{(i,k)}) \}]  \not\equiv 0 \mod m $, where $\bar{f}_{v_i}^{(i)} = f_{v_i}^{(i)} +  m\mathbb{Z}$ \@. To show the existence of other vertex labels $\bar{f}_{v_{i+1}}^{(i)},\ldots,\bar{f}_{v_n}^{(i)}$ it is sufficient to prove the existence  of $f_{v_{i+1}}^{(i)},\ldots,{f}_{v_n}^{(i)}$  such that $F^{(i)} = (0,\ldots,0,f_{v_i}^{(i)},\ldots,f_{v_n}^{(i)})$ is a flow-up class over $\mathbb Z$. By Theorem~3.8 in \cite{2019}, we know that there exist $f_{v_{i+1}}^{(i)},\ldots,{f}_{v_n}^{(i)}$. Since  $(\rho_*f^{(i)})_{v_{j}}= \rho(f_{v_{j}}^{(i)}) $, $\bar{f}_{v_{j}}^{(i)} = f_{v_{j}}^{(i)} +  m\mathbb{Z}$ also exists. This concludes the existence of $\bar{F}^{(i)}.$ 
\end{proof}



The following theorem specifies minimum flow-up generating sets for complete graphs over $\mathbb{Z}/ m\mathbb{Z}$  with edge-labels $a_i=p_1^{n_{i1}}p_2^{n_{i2}}\ldots p_k^{n_{ik}}$, where $ 0 \le n_{ij} \le m_j$ for each $j=1,\ldots,k$ and $i=1,2,\ldots, r_n$, in either increasing or decreasing order.

\begin{Th}	\label{son}
	Let $(K_n,\alpha)$ be an edge-labeled complete graph with ordered edge-labels $a_1,a_2,\ldots,a_{r_n}.$
	
	\begin{enumerate}
		\item Assume that the edge-label set  $\{a_1,a_2,\ldots,a_{r_n}\}$ is ordered with $$a_{r_n} \mid a_{r_n-1} \mid \cdots \mid a_3 \mid a_2 \mid a_1 \mid m = p_1^{m_1}p_2^{m_2} \cdots p_k^{m_k}, \quad m \neq  a_1.$$ Then the set
		{\tiny	\begin{equation*}
				B_m =	\begin{Bmatrix}
					
					\left( \begin{array}{c}
						1 \\
						1\\
						\vdots\\
						\vdots \\
						1\\
						1
					\end{array} \right) , &
					
					\left( \begin{array}{c}
						0 \\
						0\\
						\vdots\\
						0\\
						0\\
						a_1 \\
						0
					\end{array} \right), &
					
					\left( \begin{array}{c}
						0 \\
						0\\
						\vdots\\
						0\\
						a_2 \\
						0\\
						0
					\end{array} \right), &
					
					\left( \begin{array}{c}
						0 \\
						\vdots\\
						0\\
						a_4 \\
						0\\
						0\\
						0
					\end{array} \right)
					
					&,\ldots, &
					
					\left( \begin{array}{c}
						0\\
						a_{(r_{n-2}+1)} \\
						0\\
						\vdots\\
						\vdots\\
						0
					\end{array} \right), &
					
					\left( \begin{array}{c}
						a_{(r_{n-1}+1)} \\
						0\\
						\vdots\\
						\vdots\\
						0\\
						0
					\end{array} \right)
					
				\end{Bmatrix}
		\end{equation*}}
		is a minimum flow-up generating set over  $\mathbb{Z}/ m\mathbb{Z}$\@. Thus, the rank of $[\mathbb{Z}/ m\mathbb{Z}]_{(K_n,\alpha)}$ is $n$.
		
		\item Assume that the edge-label set  $\{a_1,a_2,\ldots,a_{r_n}\}$ is ordered with $$a_1 \mid a_2 \mid a_3 \mid  \cdots \mid a_{r_n -1} \mid a_{r_n} \mid m = p_1^{m_1}p_2^{m_2}\cdots p_k^{m_k}, \quad  m \neq  a_{r_n}.$$ Then the set
		{\tiny	\begin{equation*}
				B_m =\begin{Bmatrix}

					\left( \begin{array}{c}
						1 \\
						\vdots\\
						\vdots\\
						1
					\end{array} \right) , &
					
					\left( \begin{array}{c}
						a_{(r_n-(n-2))} \\
						\vdots\\
						a_{(r_n-(n-2))} \\
						a_{(r_n-(n-2))} \\
						0
					\end{array} \right)
					, &
					\left( \begin{array}{c}
						a_{(r_n-(n-3))} \\
						\vdots\\
						a_{(r_n-(n-3))} \\
						0\\
						0
					\end{array} \right) &
					
					,\ldots, &
					
					\left( \begin{array}{c}
						a_{(r_n-1)} \\
						a_{(r_n-1)} \\
						0\\
						\vdots\\
						0
					\end{array} \right)
					, &
					\left( \begin{array}{c}
						a_{r_n} \\
						0\\
						\vdots\\
						0\\
						0
					\end{array} \right)
					
				\end{Bmatrix}
		\end{equation*}}
		is a minimum flow-up generating set over  $\mathbb{Z}/ m\mathbb{Z}$\@. Thus, the rank of $[\mathbb{Z}/ m\mathbb{Z}]_{(K_n,\alpha)}$ is $n$.
	\end{enumerate}

\end{Th}



\begin{proof}
	We begin with the proof of part~(1).  The existence of
	
	{\tiny\begin{center}
			
			\begin{tabular}{c}
				$ F^{(1)} = \left( \begin{array}{c}
				1 \\
				1\\
				\vdots\\
				\vdots\\
				1\\
				1
				\end{array} \right)$ ,
				
				$	F^{(2)} = \left( \begin{array}{c}
				* \\
				*\\
				\vdots\\
				*\\
				*\\
				a_1 \\
				0
				\end{array} \right),$
				$	F^{(3)} = \left( \begin{array}{c}
				* \\
				*\\
				\vdots\\
				*\\
				a_2\\
				0 \\
				0
				\end{array} \right),$
				$	F^{(4)} = \left( \begin{array}{c}
				*\\	
				\vdots\\
				*\\
				a_4\\
				0\\
				0\\
				0
				\end{array} \right),$

				\ldots,
				
				$F^{(n-1)} = \left( \begin{array}{c}
				*\\
				a_{(r_{n-2}+1)} \\
				0\\
				\vdots\\
				\vdots\\
				0
				\end{array} \right),$
				
				$F^{(n)} = \left( \begin{array}{c}
				a_{(r_{n-1}+1)} \\
				0\\
				\vdots\\
				\vdots\\
				0\\
				0
				\end{array} \right)
				$
			\end{tabular} 
	\end{center}} 
	
	follows directly from  Theorem \ref{exist} with $a_i$ for $ i= 1,2,4,\ldots,r_{n-2}+1,r_{n-1}+1$, where $a_i$ corresponds to the smallest nonzero leading entry of some flow-up class $F^{(j)} $. Let $A^{(j)}=(0,\ldots,0,g_{v_j},\ldots,g_{v_n})$ be an arbitrary flow-up class. Then, the entry $g_{v_j}$ is a multiple of the entry $a_i$ by Theorem \ref{min}.	
	Let  $B=\{F^{(1)}, F^{(2)}, \ldots, F^{(n-1)}, F^{(n)}\}$.  Then by Theorem \ref{basis}, $B$ generates the spline module $[\mathbb{Z}/ m\mathbb{Z}]_{(K_n,\alpha)}$ as a $\mathbb{Z}$-module. Without loss of generality, for each $i$, we can set the entries in $F^{(i)}$ denoted by ``$*$'' equal to $0$. Thus, the set $B$ is equal to $B_m$. In general, the generating set $B_m$ obtained  is a minimum flow-up generating set  by Corollary \ref{criteria}.\\
	
	Part(2) can be proved in a similar way to part(1).
	
\end{proof}


\begin{figure}[h!]
	\centering
	\includegraphics[width=0.35\linewidth]{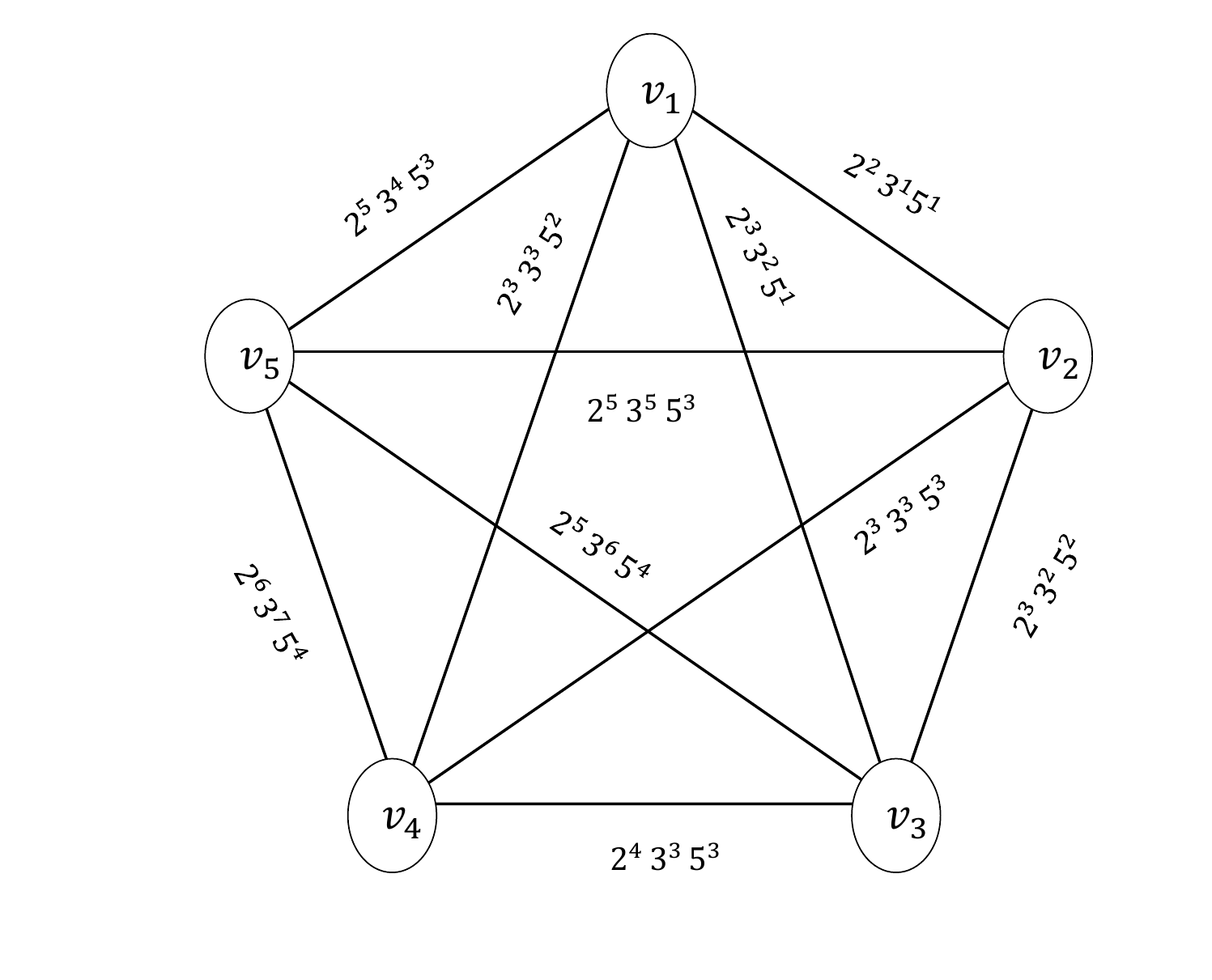}
	\caption{Edge-labeled complete graph $(K_5,\alpha)$ over $\mathbb{Z}/ (2^83^{10}5^7)\mathbb{Z}$ }
	\label{fig:k5}
\end{figure}
\begin{Ex}
	
	Let \begin{math} (K_5,\alpha) \end{math} be an edge-labeled complete graph over $\mathbb{Z}/ (2^83^{10}5^7)\mathbb{Z}$ as in Figure \ref{fig:k5}\@. By Theorem \ref{son} part(2), the set 
	{\tiny	\begin{equation*}
			B_m = \begin{Bmatrix}
				\left( \begin{array}{c}
					1 \\
					1\\
					1\\
					1\\
					1
				\end{array} \right) , &
				
				\left( \begin{array}{c}
					2^53^45^3 \\
					2^53^45^3 \\
					2^53^45^3 \\
					2^53^45^3 \\
					0
				\end{array} \right)
				, &
				\left( \begin{array}{c}
					2^53^55^3 \\
					2^53^55^3 \\
					2^53^55^3 \\
					0 \\
					0
				\end{array} \right), &
				
				\left( \begin{array}{c}
					2^53^65^4 \\
					2^53^65^4 \\
					0 \\
					0\\
					0
				\end{array} \right), &
				
				\left( \begin{array}{c}
					2^63^75^4 \\
					0 \\
					0 \\
					0\\
					0
				\end{array} \right)
				
			\end{Bmatrix}
	\end{equation*}}
	is a minimum flow-up generating set. Hence, the rank of $[\mathbb{Z}/ (2^83^{10}5^7)\mathbb{Z}]_{(K_5,\alpha)}$ is 5.
	
\end{Ex}


Let  $a$ be a zero divisor in  $\mathbb{Z}/ m\mathbb{Z}$. The following corollary specifies a minimum flow-up generating sets for splines on complete graphs over $\mathbb{Z}/ m\mathbb{Z}$ with edge-labels 	$$l_1=a^{i_1}, l_2=a^{i_2},\ldots, l_{r_n}=a^{i_{r_n}}$$   in either increasing or decreasing order. 

\begin{Cor}
	
	Let $(K_n,\alpha)$ be an edge-labeled complete graph with ordered edge-labels $a^{i_1},a^{i_2},\ldots,a^{i_{r_n}}$\@.
	
	\begin{enumerate}
		\item  We assume that the edge-label set $\{a^{i_1},a^{i_2},\ldots,a^{i_{r_n}}\}$ is ordered with $$a^{i_{r_n}} \mid a^{i_{r_n-1}} \mid \cdots  \mid a^{i_3} \mid a^{i_2} \mid a^{i_1} \mid a^k, \quad i_{r_n}\geq 1, i_1 < k.$$ Then the set
		{\tiny	\begin{equation*}
				B_{a^k} = \begin{Bmatrix}
					\left( \begin{array}{c}
						1 \\
						1\\
						\vdots\\
						\vdots\\
						1\\
						1
					\end{array} \right) , &
					
					\left( \begin{array}{c}
						0 \\
						0\\
						\vdots\\
						0\\
						0\\
						a^{i_1} \\
						0
					\end{array} \right), &
					
					\left( \begin{array}{c}
						0 \\
						0\\
						\vdots\\
						0\\
						a^{i_2} \\
						0\\
						0
					\end{array} \right), &
					
					\left( \begin{array}{c}
						0 \\
						\vdots\\
						0\\
						a^{i_4} \\
						0\\
						0\\
						0
					\end{array} \right)
					&
					,\ldots, &
					
					\left( \begin{array}{c}
						0\\
						a^{i_{(r_{n-2}+1)}} \\
						0\\
						\vdots\\
						\vdots\\
						0
					\end{array} \right), &
					
					\left( \begin{array}{c}
						a^{i_{(r_{n-1}+1)}}\\
						0\\
						\vdots\\
						\vdots\\
						0\\
						0
					\end{array} \right)
				\end{Bmatrix}
		\end{equation*}}
		is a minimum flow-up generating set over $\mathbb{Z}/ m\mathbb{Z}$. Thus, the rank of $[\mathbb{Z}/ m\mathbb{Z}]_{(K_n,\alpha)}$ is $n$.
		
		\item We assume that the edge-label set $\{a^{i_1},a^{i_2},\ldots,a^{i_{r_n}}\}$ is ordered with  $$a^{i_1} \mid a^{i_2} \mid a^{i_3} \mid \cdots \mid a^{i_{r_n-1}} \mid  a^{i_{r_n}} \mid a^k, \quad i_1\geq 1, i_{r_n} <k.$$ Then the set
		{\tiny	\begin{equation*}
				B_{a^k} = \begin{Bmatrix}
					\left( \begin{array}{c}
						1 \\
						\vdots\\
						\vdots\\
						1
					\end{array} \right) , &
					
					\left( \begin{array}{c}
						a^{i_{(r_n-(n-2))}} \\
						\vdots\\
						a^{i_{(r_n-(n-2))}} \\
						a^{i_{(r_n-(n-2))}} \\
						0
					\end{array} \right)
					, &
					\left( \begin{array}{c}
						a^{i_{(r_n-(n-3))}} \\
						\vdots\\
						a^{i_{(r_n-(n-3))}} \\
						0\\
						0
					\end{array} \right) &
					
					,\ldots, &
					
					\left( \begin{array}{c}
						a^{i_{(r_n-1)}} \\
						a^{i_{(r_n-1)}} \\
						0\\
						\vdots\\
						0
					\end{array} \right)
					, &
					\left( \begin{array}{c}
						a^{i_{r_n}} \\
						0\\
						\vdots\\
						0\\
						0
					\end{array} \right)
				\end{Bmatrix}
		\end{equation*}}
		is a  minimum flow-up generating set over $\mathbb{Z}/ m\mathbb{Z}$. Thus, the rank of $[\mathbb{Z}/ m\mathbb{Z}]_{(K_n,\alpha)}$ is $n$.
	\end{enumerate}

\end{Cor}



\begin{Cor}Let $(K_n,\alpha)$ be an edge-labeled complete graph with ordered edge-labels $a^{i_1},a^{i_2},\ldots,a^{i_{r_n}}$ so that $a^{i_{r_n}} \mid a^{i_{r_n-1}} \mid \cdots  \mid a^{i_3} \mid a^{i_2} \mid a^{i_1} \mid a^k$ with $i_{r_n}\geq 1$ and  $i_1 < k$, or vice versa. Then the rank of any connected spanning subgraph of $K_n$ is $n.$
\end{Cor}

\begin{Cor}
	Let $(K_n,\alpha)$ be an edge-labeled complete graph with ordered edge-labels $a_1,a_2,\ldots,a_{r_n}.$
	\begin{enumerate}
		\item Assume that the edge-labeled set  $\{a_1,a_2,\ldots,a_{r_n}\}$ is ordered with $$a_{r_n} \mid a_{r_n-1} \mid \cdots \mid a_1 \mid m = p_1^{m_1}p_2^{m_2} \cdots p_k^{m_k}, \quad m \neq  a_1.$$ If the spanning subgraph of $K_n$ is the wheel graph $W_n$ or the star graph $S_{n-1}$ with the central vertex labeled $v_1$, then for each of the graphs obtained the minimum generating sets is the same for $K_n$.
		
		\item Assume that the edge-labeled set  $\{a_1,a_2,\ldots,a_{r_n}\}$ is ordered with $$a_1 \mid a_2 \mid \cdots \mid a_{r_n} \mid m = p_1^{m_1}p_2^{m_2}\cdots p_k^{m_k}, \quad  m \neq  a_{r_n}.$$ If the spanning subgraph of $K_n$ is the wheel graph $W_n$ or the star graph $S_{n-1}$ with the central vertex labeled $v_n$, then for each of the  graphs obtained the minimum generating sets is the same for $K_n$.
	\end{enumerate}
\end{Cor}

The following corollary specifies  minimum flow-up generating sets for splines on complete graphs over $\mathbb{Z}/ p^t\mathbb{Z}$ with edge-labels 	$$l_1=p^{i_1}, l_2=p^{i_2},\ldots, l_{r_n}=p^{i_{r_n}}$$   in either increasing or decreasing order.

\begin{Cor}
	Let $(K_n,\alpha)$ be an edge-labeled complete graph with ordered edge-labels $p^{i_1},p^{i_2},\ldots,p^{i_{r_n}}$\@.
	
	\begin{enumerate}
		\item  Assume that the edge-label set $\{p^{i_1},p^{i_2},\ldots,p^{i_{r_n}}\}$ is ordered with $$p^{i_{r_n}} \mid p^{i_{r_n-1}} \mid \cdots \mid  p^{i_3} \mid p^{i_2} \mid p^{i_1} \mid p^t, \quad i_{r_n}\geq 1, i_1 < t.$$ Then the set
		{\tiny	 \begin{equation*}
				B_{p^k} = \begin{Bmatrix}
					\left( \begin{array}{c}
						1 \\
						1\\
						\vdots\\
						\vdots\\
						1\\
						1
					\end{array} \right) , &
					
					\left( \begin{array}{c}
						0 \\
						0\\
						\vdots\\
						0\\
						0\\
						p^{i_1} \\
						0
					\end{array} \right), &
					
					\left( \begin{array}{c}
						0 \\
						0\\
						\vdots\\
						0\\
						p^{i_2} \\
						0\\
						0
					\end{array} \right), &
					
					\left( \begin{array}{c}
						0 \\
						\vdots\\
						0\\
						p^{i_4} \\
						0\\
						0\\
						0
					\end{array} \right)
					&
					,\ldots, &
					
					\left( \begin{array}{c}
						0\\
						p^{i_{(r_{n-2}+1)}} \\
						0\\
						\vdots\\
						\vdots\\
						0
					\end{array} \right), &
					
					\left( \begin{array}{c}
						p^{i_{(r_{n-1}+1)}}\\
						0\\
						\vdots\\
						\vdots\\
						0\\
						0
					\end{array} \right)
				\end{Bmatrix}
		\end{equation*}}
		is a  minimum flow-up generating set over $\mathbb{Z}/ p^t\mathbb{Z}$. Thus, the rank of $[\mathbb{Z}/ p^t\mathbb{Z}]_{(K_n,\alpha)}$ is $n$.
		
		\item Assume that the edge-label set $\{p^{i_1},p^{i_2},\ldots,p^{i_{r_n}}\}$ is ordered with  $$p^{i_1} \mid p^{i_2} \mid p^{i_3} \mid  \cdots \mid p^{i_{r_n-1}} \mid  p^{i_{r_n}} \mid p^t, \quad i_1\geq 1, i_{r_n} <t.$$ Then the set
		
		{\tiny	\begin{equation*}
				B_{p^k} = \begin{Bmatrix}
					\left( \begin{array}{c}
						1 \\
						\vdots\\
						\vdots\\
						1
					\end{array} \right) , &
					
					\left( \begin{array}{c}
						p^{i_{(r_n-(n-2))}} \\
						\vdots\\
						p^{i_{(r_n-(n-2))}} \\
						p^{i_{(r_n-(n-2))}} \\
						0
					\end{array} \right)
					, &
					\left( \begin{array}{c}
						p^{i_{(r_n-(n-3))}} \\
						\vdots\\
						p^{i_{(r_n-(n-3))}} \\
						0\\
						0
					\end{array} \right) &
					
					,\ldots, &
					
					\left( \begin{array}{c}
						p^{i_{(r_n-1)}} \\
						p^{i_{(r_n-1)}} \\
						0\\
						\vdots\\
						0
					\end{array} \right)
					, &
					\left( \begin{array}{c}
						p^{i_{r_n}} \\
						0\\
						\vdots\\
						\\
						0
					\end{array} \right)
				\end{Bmatrix}
		\end{equation*}}
		is a minimum flow-up generating set over $\mathbb{Z}/ p^t\mathbb{Z}$. Thus, the rank of $[\mathbb{Z}/ p^t\mathbb{Z}]_{(K_n,\alpha)}$ is $n$.
	\end{enumerate}
\end{Cor}

For any positive integer $b$ we denote the largest connected edge-labeled subgraph of an edge-labeled complete graph $(K_n,\alpha)$ over $\mathbb{Z} / p^t \mathbb{Z}$ whose edges are labeled by $p^{b_i}$ , ${b\leq b_i< t}$, by $H^{(j,b)}= (V^{(j,b)},E^{(j,b)})$, where the smallest-index vertex is $v_j$.  The following theorem gives flow-up classes which form a minimum generating set for $[\mathbb{Z} / p^t \mathbb{Z}]_{(K_n,\alpha)}$.
\begin{Th}\label{thm:m=p^t}
	Let  $(K_n,\alpha)$ be an edge-labeled complete graph over $\mathbb{Z} / p^t \mathbb{Z}$ with unordered edge-labels 	$l_1=p^{i_1}, l_2=p^{i_2},\ldots, l_{r_n}=p^{i_{r_n}}$.  Then there is a flow-up class $F^{(i)}= (0,\ldots,0,f_{v_i}^{(i)},\ldots,f_{v_n}^{(i)})$ with $f_{v_i}^{(i)}=[\cup_{k=1}^{i-1}\{(p^{(i,k)}) \}]=p^{a_i}$  for $i> 1$ such that
	
	$$f_{v}^{(i)}=\left\{ \begin{array}{c c}
	p^{a_i} & \quad \text{if} \quad v \in V^{(i,a_i+1)}\\
	0 &  \quad \text{otherwise}
	
\end{array} \right.$$ Moreover, 
the rank of $[\mathbb{Z} / p^t \mathbb{Z}]_{(K_n,\alpha)}$ is $n$.
\end{Th}

\begin{proof}
	To check  whether  $F^{(i)}= (0,\ldots,0,f_{v_i}^{(i)},\ldots,f_{v_n}^{(i)})$ with given entries is a spline or not, we pick any two distinct vertices $v, w$ of $K_n$. If they are both in $V^{(i,a_i+1)}$, the difference of the corresponding vertex labelings  $f^{(i)}_{v}-f^{(i)}_{w}=p^{a_i}-p^{a_i}=0$ satisfies the spline condition. If $v$ is in $V^{(i,a_i+1)}$ but $w$ is not, the  label of the edge $vw$ is $p^a$ where $a\leq a_i$. In this case the difference $f^{(i)}_{v}-f^{(i)}_{w}=p^{a_i}$ is divisible by $p^a$. Hence the spline condition is again satisfied. If they are both not in $V^{(i,a_i+1)}$, the difference is zero.  The spline condition is again satisfied.  Hence $F^{(i)}= (0,\ldots,0,p^{a_i},*,\ldots,*)$ is a spline where ``$*$'' in $F^{(i)}$ refers  to either $p^{a_i}$ or $0$.
	
	Let $B=\{F^{(1)}, F^{(2)}, \ldots, F^{(n-1)}, F^{(n)}\}$, where

	{\tiny	\begin{center}

			\begin{tabular}{c c c c c}
				$ F^{(1)} = \left( \begin{array}{c}
				1 \\
				1\\
				\vdots\\
				1\\
				1
				\end{array} \right) $, &
				
				$F^{(2)} = \left( \begin{array}{c}
				* \\
				\vdots\\
				
				*\\
				p^{a_2} \\
				0
			\end{array} \right)$, &
			
			\ldots \;, &
			
			$F^{(n-1)} = \left( \begin{array}{c}
			*\\
			p^{a_{n-1}}\\
			0\\
			\vdots\\
			
			0
		\end{array} \right) $, &
		
		$	F^{(n)} = \left( \begin{array}{c}
		p^{a_n} \\
		0\\
		\vdots\\
		\\
		
		0
	\end{array} \right)
	.$
	\end{tabular} \end{center} }

\noindent
Note that $F^{(1)}= (1,\ldots, 1{\tiny })^T$ is a trivial spline.  Then by Theorem \ref{basis} and Corollary \ref{criteria}, $B$ is a minimum  flow-up generating set for the spline  modulo  $[\mathbb{Z}/ p^t\mathbb{Z}]_{(K_n,\alpha)}$ as a $\mathbb{Z}$-module.  It follows that 	 $\rk [\mathbb{Z} / p^t \mathbb{Z}]_{(K_n,\alpha)}=n$.

\end{proof}

\begin{figure}[h]
\centering
\includegraphics[width=0.35\linewidth]{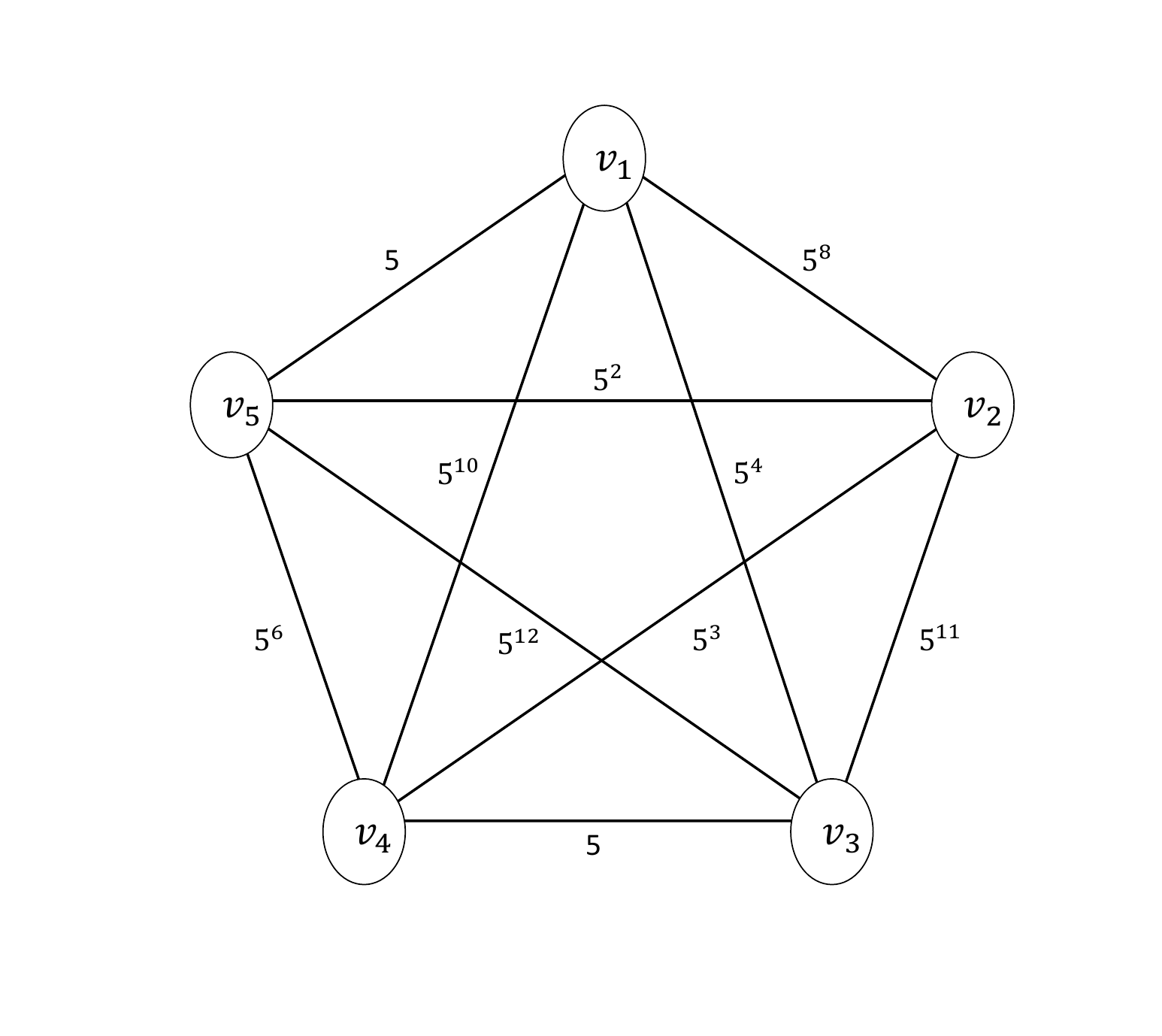}
\caption{Edge-labeled complete graph over $\mathbb{Z} / 5^{15} \mathbb{Z}$}
\label{fig:k5exunordered}
\end{figure}
\begin{Ex}

Let $(K_5,\alpha)$ be an edge-labeled complete graph over $\mathbb{Z} / 5^{15} \mathbb{Z}$ as in Figure \ref{fig:k5exunordered}. 	
Then a minimum flow-up generating set $B$ for $[\mathbb{Z} / 5^{15} \mathbb{Z}]_{(K_5,\alpha)}$ can be obtained from Theorem~\ref{thm:m=p^t} as follows: 
{\tiny	\begin{equation*}
		B= \begin{Bmatrix}
			\left( \begin{array}{c}
				1 \\
				1\\
				1\\
				1\\
				1
				
			\end{array} \right) ,
			
			\left( \begin{array}{c}
				5^8 \\
				0\\
				5^8\\
				5^8\\
				0
			\end{array} \right),
			
			\left( \begin{array}{c}
				5^{11}\\
				0\\
				5^{11}\\
				0\\
				0
			\end{array} \right),

			\left( \begin{array}{c}
				0\\
				5^{10}\\
				0\\
				0\\
				0
			\end{array} \right),
			
			\left( \begin{array}{c}
				5^{12}\\
				0\\
				0\\
				0\\
				0
			\end{array} \right)

		\end{Bmatrix}.
\end{equation*}}
\end{Ex}

\begin{Th}\label{Th:rk K_n+1}

Let $m=pq$, where $p$ and $q$ are distinct prime numbers and $(K_n,\alpha)$ be an edge-labeled complete graph over $R=\mathbb{Z} / m \mathbb{Z}$.  If the module $[\mathbb{Z} / m \mathbb{Z}]_{(K_n,\alpha)}$ has a minimum generating set consisting of the trivial spline $(1,\ldots,1)$ and  $i-1$ other flow-up splines for each $1 < i \le n$ each of whose entries is contained  in the ideal  $\langle p \rangle $ then we can either preserve the rank or increase the rank of $R_{(K_{n+1},\alpha)}$ exactly one by adding $S_n$  to $K_n$ in the following ways:
\begin{itemize}
	\item If all edge labels of $S_n$ are contained in the ideal $\langle p\rangle$, then  
	$$\rk [\mathbb{Z} / m \mathbb{Z}]_{(K_{n+1},\alpha)} = \rk  [\mathbb{Z} / m \mathbb{Z}]_{(K_n,\alpha)}+1.$$
	\item If one of the edge labels of $S_n$ is contained in the ideal $\langle q \rangle$ and the others are contained in the ideal $\langle p \rangle$, then $\rk [\mathbb{Z} / m \mathbb{Z}]_{(K_{n+1},\alpha)} = \rk  [\mathbb{Z} / m \mathbb{Z}]_{(K_n,\alpha)}$.{\tiny \label{key}}
\end{itemize}

\end{Th}

\begin{proof}
Assume that $B =	\{(1,\ldots,1), *\}	$
is a minimum flow-up generating set for $[\mathbb{Z}/ m\mathbb{Z}]_{(K_n,\alpha)}$ where ``$*$'' is $i-1$ other flow-up splines for each $1 < i \le n$ each of whose entries is contained  in the ideal  $\langle p\rangle$.
\begin{itemize}
	\item Let all edge labels of $S_n$ be contained in the ideal $\langle p \rangle$. This means that the labels of the edges are  $p$. Therefore there exists an $n$ flow-up class $F^{(n)}=(0,0,\ldots,0,f_{v_{n+1}}^{(n)} )$ with  $f_{v_{n+1}}^{(n)} =[\cup_{k=1}^{n}\{(p_i^{(n+1,k) }) \}]=p$. By construction of an edge-labeled complete graph $K_{n+1}$,
	$$ B =	\{(1,\ldots,1), *, (0,\ldots,0,p)\}	$$
	is a minimum flow-up generating set for $[\mathbb{Z}/ m\mathbb{Z}]_{(K_{n+1},\alpha)}$, where ``$*$'' is $i-1$ other flow-up splines for each $1 < i \le n$ each of whose entries is contained  in the ideal  $\langle p \rangle$. Thus, 
	$$	\rk [\mathbb{Z} / m \mathbb{Z}]_{(K_{n+1},\alpha)} = \rk  [\mathbb{Z} / m \mathbb{Z}]_{(K_n,\alpha)}+1 .$$

	\item Let one of the edge labels of $S_n$ be contained in the ideal $\langle q\rangle$ and the others are contained in the ideal $\langle p\rangle$. Then there is no $n$ flow-up class   $F^{(n)}=(0,0,\ldots,0,f_{v_{n+1}}^{(n)} )$ with  $f_{v_{n+1}}^{(n)} =[\cup_{k=1}^{n}\{(p_i^{(n+1,k) }) \}]\ne 0$ since  $[\cup_{k=1}^{n}\{(p_i^{(n+1,k) }) \}] =pq=0 \mod m.$   By construction of an edge-labeled complete graph $K_{n+1}$, 
	$$	B =	\{(1,\ldots,1), *\}	$$
	is a  minimum flow-up generating set for $[\mathbb{Z}/ m\mathbb{Z}]_{(K_{n+1},\alpha)}$, where ``$*$'' is $i-1$ other flow-up splines for each $1 < i \le n$ each of whose entries is contained  in the ideal  $\langle p\rangle$. Thus, $$ \rk [\mathbb{Z} / m \mathbb{Z}]_{(K_{n+1},\alpha)} = \rk  [\mathbb{Z} / m \mathbb{Z}]_{(K_n,\alpha)}.$$
	
\end{itemize}

\end{proof}

\begin{Ex}
We cannot say that rank always increases at one point or stays the same. Sometimes it also decreases. For example,  $\rk [\mathbb{Z} / m \mathbb{Z}]_{(K_{4},\alpha)}=2$ , but $ \rk [\mathbb{Z} / m \mathbb{Z}]_{(K_{5},\alpha)} =1$ as in  Figure~\ref{fig:k5pq}.
\begin{figure}[h]
	\centering
	\begin{subfigure}[b]{0.25\textwidth}
		\includegraphics[width=\textwidth]{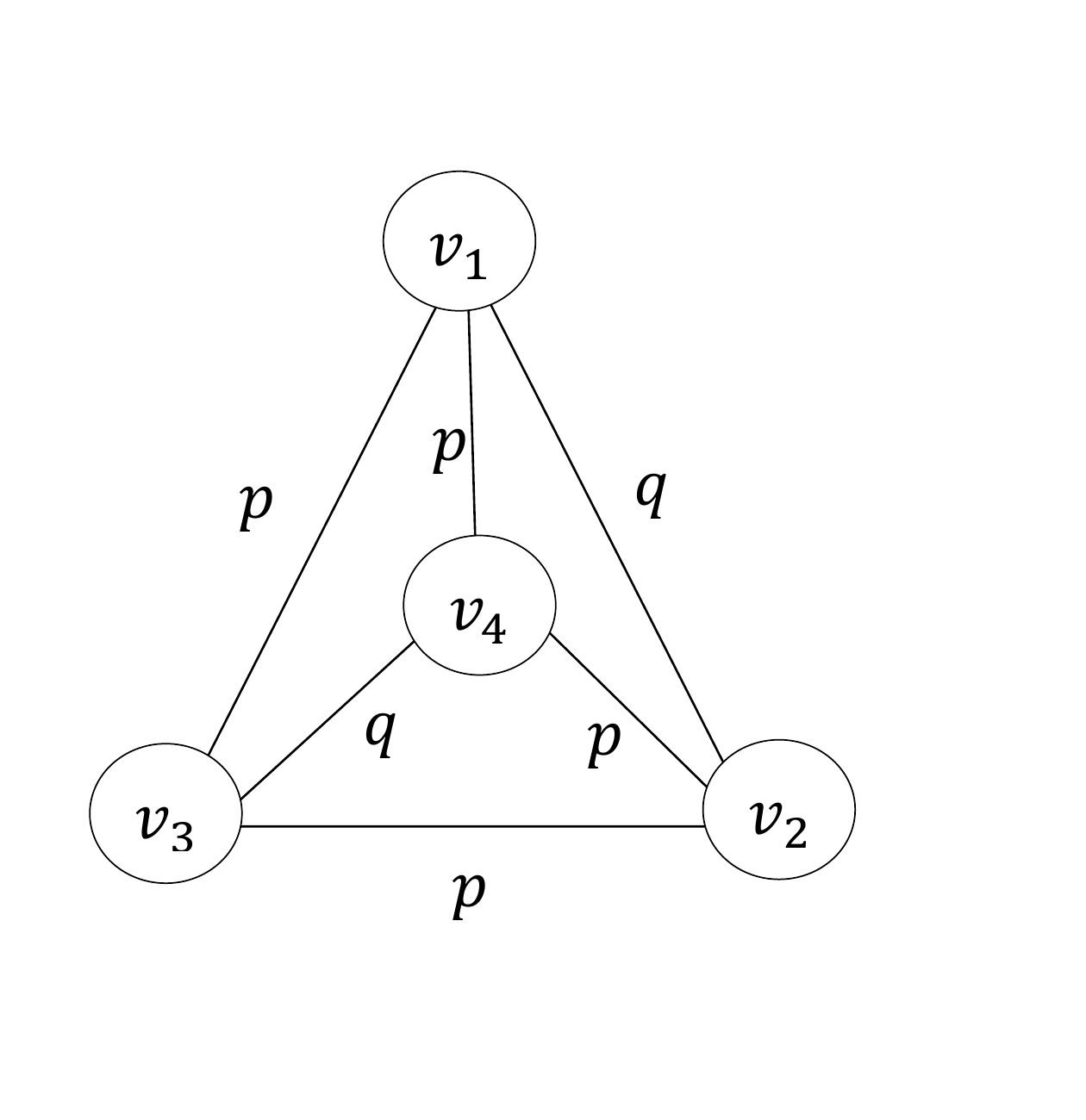}
		\label{fig:k4}
	\end{subfigure} 
	\begin{subfigure}[b]{0.35\textwidth}
		\includegraphics[width=\textwidth]{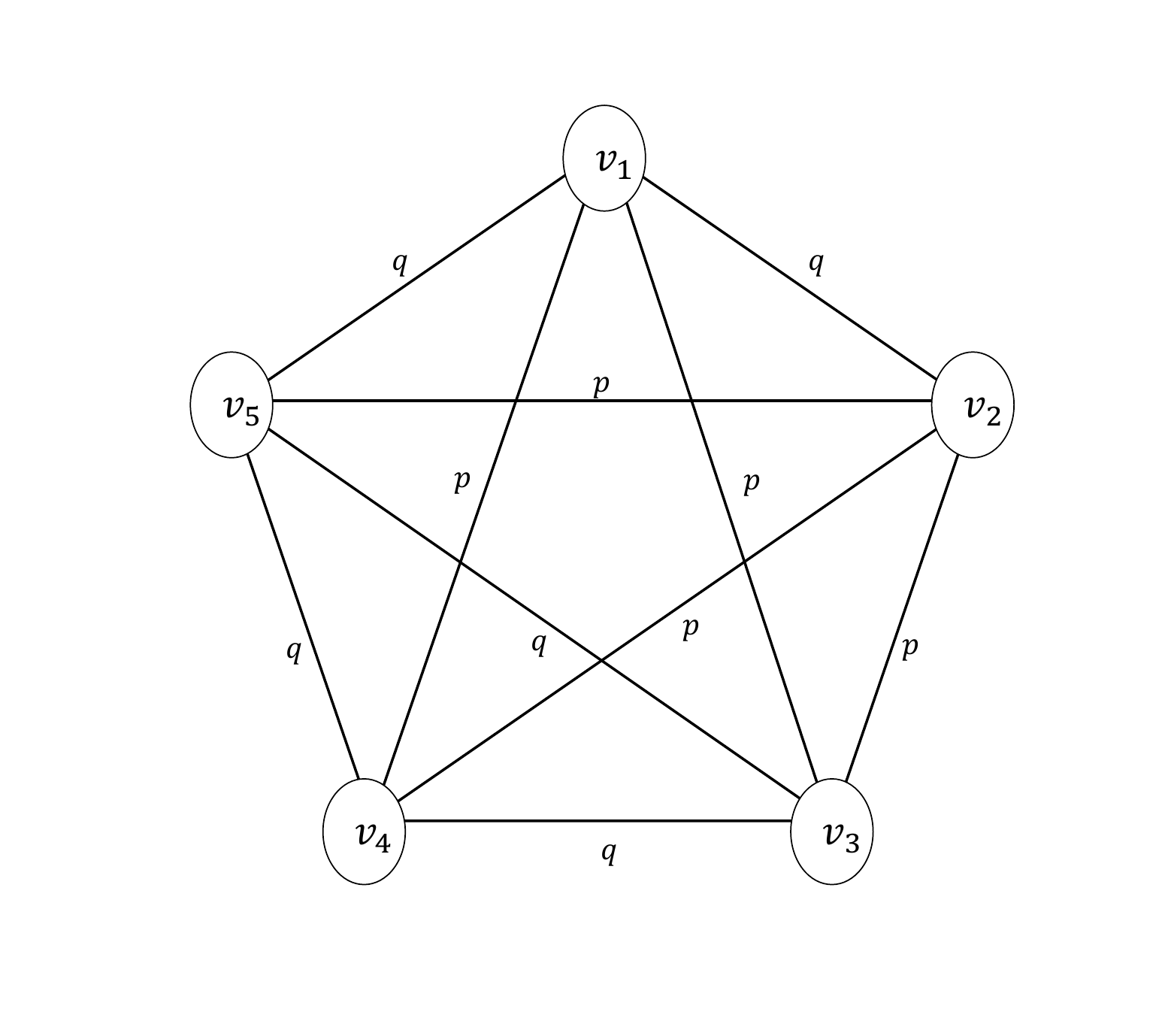}
	\end{subfigure}
	\caption{ Edge-labeled complete graphs $(K_4,\alpha)$  and $(K_5,\alpha)$ over $\mathbb{Z}/pq\mathbb{Z}$}
	\label{fig:k5pq}
\end{figure}

\end{Ex}

The following proposition is proved by using a spanning tree argument. This argument was first used by Anders et al. (see ~\cite{2018}). 

\begin{Prop} \label{rk K_n=1}
Let $m=p^\alpha q^\beta$ for some positive integers $\alpha,\beta$, where $p$ and $q$ are distinct prime numbers and $K_n$ be a complete graph. For $n \ge 4$  there exists an edge-labeled complete graph $(K_n,\alpha)$ such that  $ \rk [\mathbb{Z} / m \mathbb{Z}]_{(K_{n},\alpha)}= 1$.
\end{Prop}

\begin{proof} We construct an edge-labeled complete  graph with two spanning trees $T_1$ and $T_2$ whose edge labels are contained in the ideal $\langle p^\alpha \rangle$ and $\langle q^\beta \rangle$ respectively. If we label the edges $v_iv_{i+1}$  for $i= 1 \ldots, n-1$ with $p^\alpha$ and the others with $q^\beta$,  this gives a spanning tree $T_1$  with edges $v_iv_{i+1}$ whose edge labels are contained in the ideal $\langle p^\alpha \rangle$ and $T_2$ with edges $$v_1v_3,v_3v_5,\ldots,v_{n-3}v_{n-1},v_1v_4, v_2v_4,v_4v_6,\ldots,v_{n-2}v_n$$ when $n$ is even;
$$v_1v_3,v_3v_5,\ldots,v_{n-2}v_{n},v_1v_4, v_2v_4,v_4v_6,\ldots,v_{n-3}v_{n-1}$$ otherwise whose edge labels are contained in the ideal $\langle q^\beta \rangle$. It follows that  for every spline $F=(f_{v_1},f_{v_2},\ldots,f_{v_n})$ we have $f_{v_i} \equiv f_{v_j} \mod p^\alpha$ and $f_{v_i} \equiv f_{v_j} \mod q^\beta$ for each pair of vertices $v_i$ and $v_j$ .  Thus, $f_{v_i}- f_{v_j} \equiv 0 \mod p^\alpha q^\beta$ and so $f_{v_i} \equiv f_{v_j}  \mod m$ for all vertices $v_i$ and $v_j$. Thus, we have only trivial spline. Therefore, $\rk [\mathbb{Z} / m \mathbb{Z}]_{(K_{n},\alpha)}= 1$.

\end{proof}

\begin{Nt*}
The conclusion of Proposition \ref{rk K_n=1} does not hold for $n=3$ as in that case there are not enough edges to find two disjoint spanning trees.
\end{Nt*}

\begin{Th}\label{rk K_n}
Let $m=pq$, where $p$ and $q$ are distinct prime numbers and $K_n$ be a complete graph. For each $n \ge 2$ and each $i$ with $2 \le i \le n$ there exists an edge-labeled complete graph $(K_n,\alpha)$ over $\mathbb{Z} / m \mathbb{Z}$ with  $\rk  [\mathbb{Z} / m \mathbb{Z}]_{(K_{n},\alpha)}=i$. 

\end{Th}

\begin{proof}
We give a proof of  the theorem by induction on the number of vertices $n$.   We want to prove that there exists an edge-labeled graph $(K_n,\alpha)$ for each $i$ where $2\leq i\leq n$ such that the module $[\mathbb{Z} / m \mathbb{Z}]_{(K_{n},\alpha)}$ has a minimum generating set consisting of the trivial spline and $i-1$ other flow-up splines each of whose entries is $p$ (or $q$) . The base case is when $n=2$. In this case, $K_2$ is a path with one edge whose edge label is in the ideal $\langle p\rangle$ or $\langle q \rangle$. In either case, there exists an edge-labeled graph $(K_2,\alpha)$ such that $[\mathbb{Z} / m \mathbb{Z}]_{(K_{2},\alpha)}$  has a minimum generating set containing (1,1) and (0, p)(or (0,q)) so that the  rank is 2. Now we assume  that   the induction hypothesis holds for $n$ and prove that it holds for $n+1$.
We can construct  a complete graph $K_{n+1}$ by adding $S_n$ to $K_n$ in a way that the rank of the spline module is either fixed or increased exactly one by using Theorem~\ref{Th:rk K_n+1} so that $\rk  [\mathbb{Z} / m \mathbb{Z}]_{(K_{n+1},\alpha)}=i$ for $2 \le i \le n+1$. Hence the theorem is proved.

\end{proof}

\begin{Cor}
Let $m=pq$, where $p$ and $q$ are distinct prime numbers, and $K_n$ be a complete graph. For each $n \ge 4$ and each $i$ with $1 \le i \le n$ there exists an edge-labeled complete graph $(K_n,\alpha)$ over $\mathbb{Z} / m \mathbb{Z}$ with $\rk [\mathbb{Z} / m \mathbb{Z}]_{(K_{n},\alpha)}=i$. 

\end{Cor}

\begin{proof}
The result follows from Theorem \ref{rk K_n} and Proposition \ref{rk K_n=1}.
\end{proof}


\hfil \break

\small \thanks{\indent Selma ALTINOK, Hacettepe University Department of Mathematics, 06800 Beytepe, Ankara, Turkey\\
	\indent \,\,\, \textit{E-mail address:} \textbf{sbhupal@hacettepe.edu.tr}  \\
	\\
	\indent Gökçen DILAVER, Bursa Technical University Department of Mathematics, 16330 Yıldırım, Bursa, Turkey\\
	\indent \,\,\,\textit{E-mail address:} \textbf{gokcen.dilaver@btu.edu.tr} 
	
}

\end{document}